\documentclass[a4paper]{article}

\usepackage{amssymb}
\usepackage{amsmath}
\usepackage{amsfonts}
\usepackage{amscd}
\usepackage{amsthm}
\usepackage{color}
\usepackage{tikz}
\usepackage{pgfplots}
\usepackage{hyperref}

\newtheorem{definition}{Definition}
\newtheorem{prop}{Proposition}
\newtheorem{lemma}{Lemma}

\newtheorem{conjecture}{Conjecture}
\newtheorem{theorem}{Theorem}
\newtheorem{remark}{Remark}
\newtheorem{question}{Question}
\newtheorem{example}{Example}
\newtheorem{notation}{Notation}
\newcommand{\Q}{\mathbb{Q}}
\newcommand{\Z}{\mathbb{Z}}
\newcommand{\C}{\mathbb{C}}
\newcommand{\CuspGroup}{\mathcal{F}_1}       
\newcommand{\OurCuspGroup}{\mathcal{F}_1'}   
\newcommand{\CuspDiv}{\mathcal{D}_1}         
\newcommand{\CuspClassGroup}{\mathcal{C}_1}

\newcommand{\set}[1]{\left\lbrace #1 \right\rbrace}
\newcommand{\field}[1]{\mathbb{#1}}  
\newcommand{\halfplane}{\field H}
\newcommand{\F}{\field{F}}
\newcommand{\noarg}{\_}
\newcommand{\norm}[2]{\parallel \!\! #1 \!\! \parallel_{#2}}
\renewcommand{\P}{\field{P}}
\DeclareMathOperator{\spec}{Spec}
\DeclareMathOperator{\gon}{Gon}
\DeclareMathOperator{\mdeg}{mdeg}
\DeclareMathOperator{\supp}{Supp}
\DeclareMathOperator{\aut}{Aut}
\DeclareMathOperator{\Div}{div}
\DeclareMathOperator{\Pic}{pic}
\DeclareMathOperator{\PSL}{PSL}
\definecolor{green}{rgb}{0,0.7,0}

\title{Gonality of the modular curve $X_1(N)$}
\author{Maarten Derickx\footnote{Mathematisch Instituut, Universiteit Leiden, Postbus 9512, 2300 RA Leiden, The Netherlands.
	E-mail: {\sf maarten@mderickx.nl}.} \ \ and \ Mark van Hoeij\footnote{Department of Mathematics, 
        Florida State University, Tallahassee, Florida 32306, USA.
        E-mail: {\sf hoeij@math.fsu.edu}. Supported by NSF grants 1017880 and 1319547.}}
\begin{document}
\maketitle

\begin{abstract}
In this paper we compute the gonality over $\Q$ of the modular curve $X_1(N)$ for all $N \leqslant 40$ and give upper bounds for each $N \leqslant 250$. This allows us to determine all $N$
for which
$X_1(N)$ has infinitely points of degree $\leqslant 8$. We conjecture that the modular units of $\Q(X_1(N))$
are freely generated by $f_2,\ldots,f_{\lfloor N/2 \rfloor +1}$ where $f_k$ is 
obtained from the equation for $X_1(k)$.
\end{abstract}

\section{Introduction}
\begin{notation} If $K$ is a field, and $C/K$ is a curve\footnote{In this paper, a {\em curve} over a field $K$
is a scheme, projective and smooth of relative dimension 1 over $\spec K$ that is 
geometrically irreducible.}, then $K(C)$ is the function field
of $C$ over $K$. The gonality ${\rm Gon}_{K}( C )$ is ${\rm min} \{ {\rm deg}(f) \, | \, f \in K(C)-K \}$.
In this article we are interested in the case $C = X_1(N)$, and $K$ is either $\Q$ or $\mathbb{F}_p$.
\end{notation}

It was shown in \cite{scriptie} that if $C/\Q$ is a curve and $p$ is a prime of good reduction of then:
\begin{equation}
	{\rm Gon}_{\mathbb{F}_p}(C)  \leqslant  {\rm Gon}_{\Q}(C).    \label{CompareGon}
\end{equation}
A similar statement was given earlier in \cite{frey} which attributes it to \cite{deuring}.
We use~(\ref{CompareGon}) only for $C = X_1(N)$. The primes of good reduction of $X_1(N)$ are the primes $p\nmid N$.


	%
The main goal in this paper is to compute ${\rm Gon}_{\Q}(X_1(N))$ for $N \leqslant 40$.
The $\Q$-gonality for $N \leqslant 22$ was already known \cite[p.~2]{sutherland},
so the cases $23 \leqslant N \leqslant 40$ are of most interest.
For each $N$, it suffices to:
\begin{itemize}
\item Task 1:  Compute a basis of ${\rm div}(\CuspGroup(N))$, which denotes the set of divisors of modular units over $\Q$, see 
Definition~\ref{DefCuspGroup} in Section~\ref{CuspFunctions} for details.
\item Task 2:  Use LLL techniques to search ${\rm div}(\CuspGroup(N))$ for the divisor of a non-constant function $g_N$ of lowest degree.
\item Task 3:  Prove (for some prime $p \nmid N$) that $\mathbb{F}_p(X_1(N)) - \mathbb{F}_p$
has no elements
of degree $< {\rm deg}(g_N)$.
Then~(\ref{CompareGon}) implies that the $\Q$-gonality is ${\rm deg}(g_N)$. \\
\end{itemize}
\mbox{} \ \ {\bf Table 1:} ${\rm Gon}_{\Q}(X_1(N))$ for $N \leqslant 40$. \ \ Upper bounds for $N \leqslant 250$. \\[3pt]
\begin{tabular}{|c|c|c|c|c|c|c|c|c|c|c|}
\hline
$N$ & 1 & 2 & 3 & 4 & 5 & 6 & 7 & 8 & 9 & 10 \\[-1pt]
gon $=$ & 1 & 1 & 1 & 1 & 1 & 1 & 1 & 1 & 1 & 1 \\ \hline   
$N$ & 11 & 12 & 13 & 14 & 15 & 16 & 17 & 18 & 19 & 20 \\[-1pt]
gon $=$ & 2 & 1 & 2 & 2 & 2 & 2 & 4 & 2 & 5 & 3 \\ \hline
$N$ & 21 & 22 & 23 & 24 & 25 & 26 & 27 & 28 & 29 & 30 \\[-1pt]
gon $=$ & 4 & 4 & 7 & 4 & 5 & 6 & 6 & 6 & 11 & 6 \\ \hline
$N$ & 31 & 32 & 33 & 34 & 35 & 36 & 37 & 38 & 39 & 40 \\[-1pt]
gon $=$ & 12 & 8 & 10 & 10 & 12 & 8 & 18 & 12 & 14 & 12 \\ \hline
$N$             & 41 & 42 & 43 & 44 & 45 & 46 & 47 & 48 & 49 & 50 \\[-1pt]
gon $\leqslant$ & 22 & 12 & 24 & 15 & 18 & 19 & 29 & 16 & 21 & 15 \\ \hline
$N$             & 51 & 52 & 53 & 54 & 55 & 56 & 57 & 58 & 59 & 60 \\[-1pt]
gon $\leqslant$ & 24 & 21 & 37 & 18 & 30 & 24 & 30 & 31 & 46 & 24  \\ \hline
$N$             & 61 & 62 & 63 & 64 & 65 & 66 & 67 & 68 & 69 & 70 \\[-1pt]
gon $\leqslant$ & 49 & 36 & 36 & 32 & 42 & 30 & 58 & 36 & 44 & 36 \\ \hline
$N$             & 71 & 72 & 73 & 74 & 75 & 76 & 77 & 78 & 79 & 80 \\[-1pt]
gon $\leqslant$ & 66 & 32 & 70 & 51 & 40 & 45 & 60 & 42 & 82 & 48 \\ \hline
$N$             & 81 & 82 & 83 & 84 & 85 & 86 & 87 & 88 & 89 & 90 \\[-1pt]
gon $\leqslant$ & 54 & 58 & 90 & 48 & 72 & 64 & 70 & 60 & 104& 48 \\ \hline
$N$             & 91 & 92 & 93 & 94 & 95 & 96 & 97 & 98 & 99 & 100 \\[-1pt]
gon $\leqslant$ & 84 & 66 & 80 & 83 & 90 & 56 & 123 & 63 & 90 & 60 \\ \hline
$N$             & 101 & 102 & 103 & 104 & 105 & 106 & 107 & 108 & 109 & 110 \\[-1pt]
gon $\leqslant$ & 133 & 72 & 139 & 84 & 96 & 105 & 150 & 72 & 156 & 90 \\ \hline
$N$             & 111 & 112 & 113 & 114 & 115 & 116 & 117 & 118 & 119 & 120 \\[-1pt]
gon $\leqslant$ & 114 & 96 & 167 & 90 & 132 & 105 & 126 & 120 & 144 & 96 \\ \hline
$N$             & 121 & 122 & 123 & 124 & 125 & 126 & 127 & 128 & 129 & 130 \\[-1pt]
gon $\leqslant$ & 132 & 139 & 140 & 120 & 125 & 96 & 211 & 112 & 154 & 126 \\ \hline
$N$             & 131 & 132 & 133 & 134 & 135 & 136 & 137 & 138 & 139 & 140 \\[-1pt]
gon $\leqslant$ & 225 & 120 & 180 & 156 & 144 & 144 & 246 & 132 & 253 & 144 \\ \hline
$N$             & 141 & 142 & 143 & 144 & 145 & 146 & 147 & 148 & 149 & 150 \\[-1pt]
gon $\leqslant$ & 184 & 189 & 210 & 128 & 210 & 184 & 168 & 171 & 291 & 120 \\ \hline
$N$             & 151 & 152 & 153 & 154 & 155 & 156 & 157 & 158 & 159 & 160 \\[-1pt]
gon $\leqslant$ & 299 & 180 & 216 & 180 & 240 & 168 & 323 & 234 & 234 & 184 \\ \hline
$N$             & 161 & 162 & 163 & 164 & 165 & 166 & 167 & 168 & 169 & 170 \\[-1pt]
gon $\leqslant$ & 264 & 162 & 348 & 210 & 240 & 240 & 365 & 192 & 260 & 216 \\ \hline
$N$             & 171 & 172 & 173 & 174 & 175 & 176 & 177 & 178 & 179 & 180 \\[-1pt]
gon $\leqslant$ & 270 & 231 & 392 & 210 & 240 & 240 & 290 & 274 & 420 & 192 \\ \hline
$N$             & 181 & 182 & 183 & 184 & 185 & 186 & 187 & 188 & 189 & 190 \\[-1pt]
gon $\leqslant$ & 429 & 252 & 310 & 264 & 342 & 240 & 360 & 276 & 288 & 270 \\ \hline
$N$             & 191 & 192 & 193 & 194 & 195 & 196 & 197 & 198 & 199 & 200 \\[-1pt]
gon $\leqslant$ & 478 & 224 & 488 & 328 & 336 & 252 & 508 & 240 & 519 & 240 \\ \hline
$N$             & 201 & 202 & 203 & 204 & 205 & 206 & 207 & 208 & 209 & 210 \\[-1pt]
gon $\leqslant$ & 374 & 382 & 420 & 288 & 420 & 398 & 396 & 336 & 450 & 288 \\ \hline
$N$             & 211 & 212 & 213 & 214 & 215 & 216 & 217 & 218 & 219 & 220 \\[-1pt]
gon $\leqslant$ & 583 & 351 & 420 & 396 & 462 & 288 & 480 & 445 & 444 & 360 \\ \hline
$N$             & 221 & 222 & 223 & 224 & 225 & 226 & 227 & 228 & 229 & 230 \\[-1pt]
gon $\leqslant$ & 504 & 342 & 651 & 384 & 360 & 444 & 675 & 360 & 687 & 396 \\ \hline
$N$             & 231 & 232 & 233 & 234 & 235 & 236 & 237 & 238 & 239 & 240 \\[-1pt]
gon $\leqslant$ & 480 & 420 & 711 & 336 & 552 & 435 & 520 & 432 & 748 & 384 \\ \hline
$N$             & 241 & 242 & 243 & 244 & 245 & 246 & 247 & 248 & 249 & 250 \\[-1pt]
gon $\leqslant$ & 761 & 396 & 486 & 465 & 504 & 420 & 630 & 480 & 574 & 375 \\ \hline
\end{tabular}
\\
\noindent Tasks 1--3 are only possible when:
\begin{enumerate}
\item[(a)] There is a modular unit $g_N$ of degree ${\rm Gon}_{\Q}(X_1(N))$.
\item[(b)] There is a prime $p \nmid N$ for which ${\rm Gon}_{\mathbb{F}_p}(X_1(N)) = {\rm Gon}_{\Q}(X_1(N))$.
\end{enumerate}

We have completed Tasks 1--3 for $1 < N \leqslant 40$, and hence (a),(b) are true in this range.
We do not know if they hold in general.

We implemented two methods for Task~1. Our webpage~\cite{URLgonality} gives the resulting
basis of ${\rm div}(\CuspGroup(N))$ for $N \leqslant 300$.
For Task~2, for each $4 \leqslant N \leqslant 300$ we searched ${\rm div}(\CuspGroup(N))$ for short\footnote{We want vectors with small 1-norm
because ${\rm deg}(g) = \frac12 || {\rm div}(g) ||_1.$}
vectors, and placed the best function we found, call it $g_N$, on our webpage~\cite{URLgonality}.
The degree of any non-constant function is by definition an upper bound for the gonality.
Table~1 gives ${\rm deg}(g_N)$ for $N \leqslant 250$.

Finding the shortest vector in a $\Z$-module is NP-hard. For large $N$, this forced us to
resort to a probabilistic search (we randomly scale our vectors, apply an LLL search, and repeat).
So we can not prove that every $g_N$ on our webpage is optimal, even if we assume (a).

For certain $N$ (e.g. $N=p^2$, see Section~\ref{Patterns})
there are other ways of finding functions of low degree.
Sometimes a good function can be found in a subfield of $\Q(X_1(N))$ over $\Q(X_1(1))$, see \cite{URLgonality}.
All low degree functions we found with these methods were also found by
our probabilistic LLL search. So the upper bounds in Table~1 are likely sharp when (a) holds (Question~\ref{vraag} in Section~\ref{Task1}).

At the moment, our only method to prove that an upper bound is sharp
is to complete Task~3, which we have done for $N \leqslant 40$.
The computational cost of Task~3 increases drastically as a function of the gonality.
Our range $N \leqslant 40$ contains gonalities that are much higher
than the previous record, so in order to perform Task~3 for all $N \leqslant 40$ it
was necessary to introduce several new computational ideas.

Upper bounds (Tasks~1 and~2) will be discussed in Section~\ref{CuspFunctions}, and lower bounds
(Task~3) in Section~\ref{lower}.
We cover $N=37$ separately (Theorem~\ref{N37}), this case is the most work
because it has the highest gonality in our range $N \leqslant 40$.
Sharp lower bounds for other $N \leqslant 40$ can be obtained with the same ideas.  Our
computational proof (Task~3) for each $N \leqslant 40$ can be verified by downloading the Magma files from~\cite{URLgonality}.

\begin{remark} \label{sharpCompare}
For each $N \leqslant 40$, the 
$\Q$-gonality happened to be the ${\mathbb{F}_p}$-gonality
for the smallest prime $p \nmid N$.
That was fortunate because the computational complexity of Task~3 depends on $p$.

We can not expect the ${\mathbb{F}_p}$-gonality to equal the $\Q$-gonality for every $p$.
For example, consider the action of diamond operator $<\hspace{-3pt}12\hspace{-3pt}>$
on $\C(X_1(29))$.
The fixed field has index 2
and genus 8
(type: \verb+GammaH(29,[12]).genus()+ in Sage).
By Brill-Noether theory, this subfield contains a function $f_{\rm BN}$ of degree $\leqslant \lfloor (8+3)/2 \rfloor = 5$.
Viewed as element of $\C(X_1(29))$, its degree is $\leqslant 2 \cdot 5$ which is less than the
$\Q$-gonality\footnote{We do not know if there are other $N \leqslant 40$ with $\C$-gonality $\neq$ $\Q$-gonality.} 11.
By Chebotarev's theorem, there must then be a positive density of primes $p$
for which the $\mathbb{F}_p$-gonality of $X_1(29)$ is less than 11. 
\end{remark}


\section{Modular equations and modular units}
\label{CuspFunctions}
\begin{definition}
\label{DefCuspGroup}
A non-zero element of $\Q(X_1(N))$ is called a {\em modular unit} (see \cite{kubertlang}) when all its poles and roots are cusps.
Let $\CuspGroup(N) \subset \Q(X_1(N))^* / \Q^*$ be the group of {modular units} mod $\Q^*$.

There are $\lfloor N/2 \rfloor +1$  ${\rm Gal}(\overline{\Q}/\Q)$-orbits of cusps,
denoted\footnote{Let $d|N$, $0 \leqslant i < d$, with gcd$(i,d)=1$ and let $j$ be such that the point $P_{d,i,j} = (i,\zeta_N^j)$
has order $N$ in the Neron $d$-gon $\Z/d\Z \times \mathbb G_m$.  Let $C_{d,i,j}$ be the cusp corresponding to $P_{d,i,j}$,
then $C_{d,i,j}$ and $C_{d',i',j'}$ are in the same Galois orbit iff $d=d'$ and $i \equiv \pm i'$ mod $d$. We denote the Galois orbit
of $C_{d,i,j}$ as $C_n$ where $0 \leqslant n \leqslant N/2$ and $n \equiv \pm iN/d$ mod $N$. With this numbering, the diamond operator
$<\hspace{-3pt}i\hspace{-3pt}>$ sends $C_n$
to $C_{n'}$ where $n' \equiv \pm ni$ mod $N$.} as $C_0,\ldots,C_{\lfloor N/2 \rfloor}$. Let
\[ \CuspDiv(N) := \Z C_0 \oplus \cdots \oplus \Z C_{\lfloor N/2 \rfloor} \]
be the set of $\Q$-rational cuspidal divisors.
The degree\footnote{The degree of $C_i$ is as follows.  Let $d = {\rm gcd}(i,N)$.
If $i \in \{0,N/2\}$ then ${\rm deg}(C_i) = \lceil \phi(d)/2 \rceil$, otherwise ${\rm deg}(C_i) = \phi(d)$, where $\phi$ is Euler's function.}
of $\sum n_i C_i$ is $\sum n_i {\rm deg}(C_i)$.
Denote $\CuspDiv^0(N)$ as the set of cusp-divisors of degree 0, and
\[ \CuspClassGroup(N) := \CuspDiv^0(N) / {\rm div}(\CuspGroup(N)), \] a finite group called the {\em cuspidal class group}.
\end{definition}

Let $E$ be an elliptic curve over a field $K$, and $P$ be a point on $E$ of order exactly $N$.
If $N \geqslant 4$ and ${\rm char}(K) \nmid N$,
one can represent the pair $(E,P)$ in {\em Tate normal form:}
\begin{equation}
\label{bc}
	Y^2+(1-c)XY-bY = X^3-bX^2, {\rm \ \ with \ the \ point \ } (0,0).
\end{equation}
This representation is unique and hence $b,c$ are functions on pairs $(E,P)$.
The function field $K(X_1(N))$ is generated by $b,c$.
Whenever we use the notation $b$ or $c$, we implicitly assume $N \geqslant 4$,
because the reduction to~(\ref{bc}) succeeds if and only if $N \geqslant 4$.
This implies (for $N \geqslant 4$) that poles of $b,c$ must be cusps.
The discriminant of~(\ref{bc}) is $\Delta := b^3 \cdot (16b^2+(1-20c-8c^2)b+c(c-1)^3)$
so $E$ degenerates when $\Delta=0$.
So all roots of $\Delta$ (and hence of $b$) are cusps.
Poles of $\Delta, b$ are cusps because poles of $b,c$ are cusps.
So $\Delta,b$ are modular units, and hence
\[ F_2 := b^4/\Delta  \  = \  \frac{b}{16b^2+(1-20c-8c^2)b+c(c-1)^3}  \ \ \ {\rm and} \ \ \ F_3 := b
\]
are modular units as well.

For $N \geqslant 4$, the functions $b,c$ on $X_1(N)$ satisfy a polynomial equation $F_N \in \Z[b,c]$,
namely (for $N=4,5,6,7,\ldots$) \  $c$, \ $b-c$, \ $c^2+c-b$, \ $b^2-bc-c^3, \ldots$

If $k \neq N$, 
the condition that the order of $P$ is $k$ is incompatible with the condition that the order is $N$.
This, combined with the observation that all poles of $b,c$ are cusps, implies (for $N,k \geqslant 4$)
that the {\em modular equation} $F_k$ is a {\em modular unit} for $X_1(N)$.
We define a subgroup of $\CuspGroup(N)$ generated by {\em modular equations}\footnote{$F_k$ is a modular equation for $X_1$
if it corresponds to $P$ having order $k$.  A computation is needed to show that $F_2$, $F_3$
are modular equations in this sense. The fact that $F_2$ and $F_3$ correspond to order 2 and 3 is obscured by the $b,c$ coordinates,
so we introduce $j,x_0$ coordinates for $X_1(N)$ that apply to any $N>1$ provided that $j \not\in \{0,1728\}$.
Here $x_0$ is the $x$-coordinate of a point $P$ on $y^2 = 4x^3 - 3j(j-1728) x - j(j-1728)^2$.
The condition that $P$ has order 2 or 3 can be expressed with equations $\tilde{F}_2, \tilde{F}_3 \in \Q[j,x_0]$.
These $\tilde{F}_2, \tilde{F}_3$ are functions on $X_1(N)$ for any $N>1$. Hence they can (for $N>3$)
be rewritten to $b,c$ coordinates. To obtain modular units, we have to ensure that all poles and roots are cusps,
which requires an adjustment: $F_2 := \tilde{F}_2^2/(j^2(j-1728)^3)$ and $F_3 := \tilde{F}_3^3/\tilde{F}_2^4$.}:
\[  \OurCuspGroup(N) := \, <\hspace{-2pt}F_2,F_3,\ldots, F_{\lfloor N/2 \rfloor + 1}\hspace{-2pt}> \  \subseteq \, \CuspGroup(N). \]
\vspace{1pt}

\begin{conjecture} \label{klomp}
$\OurCuspGroup(N) = \CuspGroup(N)$ for $N \geqslant 3$.
In other words, $\CuspGroup(N)$ is
freely generated by modular equations $F_2, \ldots, F_{\lfloor N/2 \rfloor + 1}$.
\end{conjecture}
We verified this for $N \leqslant 100$, see also Section~\ref{CompRemarks}.
The conjecture holds for $N=3$ because $F_2$ rewritten to $j,x_0$ coordinates
generates $\CuspGroup(3)$. The case $N=2$ is a little different, clearly $F_2$ can not generate $\CuspGroup(2)$ since it must
vanish on $X_1(2)$.  However, rewriting $F_2 F_4$ to $j,x_0$ coordinates produces a generator for $\CuspGroup(2)$.
The conjecture is only for $\Q$; if ${X_1(N)_K}$ has more than $\lfloor N/2 \rfloor + 1$ Galois orbits of cusps,
for example $X_1(5)_{K}$ with $K=\C$ or $K=\F_{11}$, then the rank of $\OurCuspGroup(N)$ would be too low.


\subsection{Computations} \label{CompRemarks}
As $N$ grows, the size of $F_N$ grows quickly.
Sutherland \cite{sutherland} obtained smaller equations by replacing $b,c$
with other generators of the function field. For $6 \leqslant N \leqslant 9$, use $r,s$
defined by \[   r = \frac{b}{c}, \ s = \frac{c^2}{b-c}, \ \  b = r s (r - 1), \ c = s (r - 1)  \]
and for $N \geqslant 10$, use $x,y$ defined by
\[ x = \frac{s-r}{rs-2r+1}, \ y = \frac{rs-2r+1}{s^2-s-r+1}, \ \ r = \frac{x^2y-xy+y-1}{x(xy-1)}, \ s = \frac{xy-y+1}{xy}. \]
The polynomial defining $X_1(N)$ is then written as $f_4 := c$, $f_5 := b-c$, $f_6 := s-1$,
$f_7 := s-r$, $f_8 := rs-2r+1$, $f_9 := s^2-s-r+1$,
$f_{10} := x - y + 1$, $f_{11} := x^2y-xy^2+y-1$, $f_{12} := x-y$,
$f_{13} := x^3y-x^2y^2-x^2y+xy^2-y+1$, etc.   Explicit expressions for $f_{10},\ldots,f_{189} \in \Z[x,y]$
can be downloaded from Sutherland's website \url{http://math.mit.edu/~drew/X1_altcurves.html}.

The same website also lists upper bounds for the gonality for $N \leqslant 189$, that are
often sharp when $N$ is prime.
Table~1 improves this bound for every composite $N > 26$, a few composite $N < 26$, but only three primes: $31$, $67$, and $101$.
When $N$ is prime, we note that Sutherland's \cite{sutherland} bound, ${\rm deg}(x)$, equals $[ 11 N^2/840 ]$ where the brackets denote rounding
to the nearest integer ($[11 N^2/840 ]$ is a valid upper bound for any $N>6$, but it is not very sharp for composite $N$'s).

Let $f_2 := F_2$ and $f_3 := F_3$.
Then $F_k/f_k \in \, <\hspace{-2pt}f_2,\ldots,f_{k-1}\hspace{-2pt}>$ for each $k \geqslant 2$.  In particular
\[ \OurCuspGroup(N) = \, <\hspace{-2pt}f_2, f_3, \ldots f_{\lfloor N/2 \rfloor +1}\hspace{-2pt}>. \]

For each $3 \leqslant N \leqslant 300$ and $2 \leqslant k \leqslant \lfloor N/2 \rfloor + 1$
we calculated ${\rm div}(f_k) \in \CuspDiv(N)$. This data can be downloaded (in row-vector notation) from
our webpage \cite{URLgonality}.
This data allows one to determine $\CuspDiv^0(N) / {\rm div}( \OurCuspGroup(N) )$ for $N \leqslant 300$.
If that
is $\cong \CuspClassGroup(N)$, then the conjecture holds for $N$. We tested this by computing $\CuspClassGroup(N)$ with
Sage\footnote{The $\Z$-module of modular units is computed with modular symbols by determining
the $\sum n_ic_i \in \Z^{\rm cusps}$ of degree $0$ with $\sum n_i \lbrace c_i,\infty\rbrace \in H_1(X_1(N)(\C),\Z) \subset H_1(X_1(N)(\C),\Q)$.}
for $N \leqslant 100$.
The ${\rm div}(f_k)$-data has other applications as well:
\begin{example}
Let $N = 29$. Suppose one wants to compute explicit generators for
the subfield of index 2 and genus 8 mentioned in Remark~\ref{sharpCompare}.
Let $\tilde{x},\tilde{y}$ denote the images of $x,y$ under the diamond operator $<\hspace{-3pt}12\hspace{-3pt}>$.
Clearly $\tilde{x}x, \, \tilde{y}y$ are in our subfield, which raises the question: How to compute $\tilde{x},\tilde{y}$?

Observe that $x = f_7/f_8$ and $y = f_8/f_9$ (The relations $1-x = f_5 f_6/(f_4f_8)$, \ $1-y = f_6 f_7/f_9$, \  $1-xy = f_6^2/f_9$
may be helpful for other examples.)  So we can find ${\rm div}(x)$ by subtracting the (7-1)'th and (8-1)'th row-vector
listed at \cite{URLgonality} for $N=29$.
We find $(0, -1, -2, -3, -1, 0, 0, 0, 3, 2, -1, -3, 2, 3, 1)$ which encodes ${\rm div}(x) =$
\[  -C_1 -2C_2 - 3C_3 - C_4 + 3 C_8 + 2C_9 - C_{10} -3 C_{11} + 2 C_{12} + 3 C_{13}  + C_{14}.  \]
The diamond operator $<\hspace{-3pt}12\hspace{-3pt}>$ sends $C_i$ to $C_{\pm 12i \ {\rm mod}\,N}$
and hence ${\rm div}( \tilde{x} ) =$
\[ 2 C_1 - C_4 -2 C_5 +C_6 -3C_7+ 2C_8+ 3C_9 - C_{10}+ 3C_{11} -C_{12} -3 C_{13}. \]
Since ${\rm div}(f_2),\ldots,{\rm div}(f_{15})$ are listed explicitly at \cite{URLgonality}, solving linear equations provides
$n_2,\ldots,n_{15}$ for which ${\rm div}( \tilde{x} ) = \sum n_i {\rm div}( f_i )$.
Setting $g := \prod f_i^{n_i} =$
\[{\frac { ( {x}^{2}y-xy+y-1 )  ( x-1 ) ^{2} ( x-y+1 )  ( {x}^{2}y-x{y}^{2}-{x}^{2}
+xy-x+y-1 ) ^{4}{y}^{3}}{ ( y-1 ) ^{2} ( xy-1 )  ( x-y )  ( {x}^{2}y-x{y}^{
2}-xy+{y}^{2}-1 ) ^{4}{x}^{4}}}, \]
it follows that $\tilde{x} = cg$ for some constant $c$
($c$ is not needed here, but it can be determined easily by evaluating $\tilde{x}$ and $g$ at a point.)
Repeating this computation for $y$, we find explicit expressions for $\tilde{x}x, \, \tilde{y}y$.
An algebraic relation can then be computed with resultants; it turns out that $\tilde{x}x, \, \tilde{y}y$ generate the subfield.
\end{example}

\subsection{Explicit upper bound for the gonality for $N \leqslant 40$} \label{Task1}
The following table lists for each $10 < N \leqslant 40$ a function of minimal degree.
We improve the upper bound from Sutherland's website (mentioned in the previous section)
in 16 out of these 30 cases. \\[5pt]
\begin{tabular}{|c|c|c|c|c|c|c|c|c|c|c|c|c|c|c|}
\hline
          11     &  12      &    13    &   14     &    15    &  16     &    17   &    18       &     19   &     20 &  21     &  22      &    23    &   24     &    25  \\
          $x$    &  $x$     &    $x$   &   $x$    &    $x$   &  $y$    &    $x$  &   $h_1$     &     $x$  &    $x$& $h_1$   &  $x$     &    $x$   &  $h_1$   &  $h_2$  \\ \hline
          26     &    27   &    28       &     29   &     30    &  31     &  32      &    33    &   34     &    35    &  36     &    37   &    38       &     39   &     40     \\
          $y$    &   $h_3$ &  $h_3$      &     $x$  &   $h_5$   &  $h_1$  &  $h_4$   &   $h_6$  & $h_1$    &  $h_7$   &  $h_8$  &   $x$   &   $h_2$     &   $h_9$  &   $h_5$    \\
\hline
\end{tabular} \\[3pt]

\noindent Here
\[ h_1 =  \frac{x^2 y-x y^2+y-1}{(x-y)x^2y}, \ \ 
h_2 =  \frac{x (1-y) (x^2 y-x y^2-x y+y^2-1)}{ (x-y+1) (x^2 y-x y^2+y-1) }, \]
\[ h_3 =  \frac{(1-x) (x^2 y-x y^2-x y+y^2-1)}{(x-y) (x^2 y-x y^2+y-1)}, \ \ 
h_4 =  \frac{(1-x) (x^2 y-x y^2+y-1)}{x (1-y)}, \]
\[ h_5 =  \frac{(1-y) (x^2 y-x y^2-x y+y^2-1)}{(x-y) y (x-y+1)} \]
\[
h_6 = \frac{f_{10}f_{11}f_{12}}{f_{17}}, \ \ 
h_7 = \frac{f_{17}}{f_{18}}, \ \
h_8 = \frac{f_{14}f_{17}^2}{f_{19}^2}, \ \ 
h_9 = \frac{f_{12}f_{13}f_{14}}{f_{19}}. \]
Each $h_1,\ldots,h_9$ is in the multiplicative group $<\hspace{-2pt}f_2,f_3,\ldots\hspace{-2pt}>$.
To save space, we only spelled out $h_1,\ldots,h_5$ in $x,y$-notation  (the $f_{19}$ that appears in $h_9$
is substantially larger than the $f_{11}$ that appears in $h_1$).
Similar expressions for $N \leqslant 300$ are given on our website \cite{URLgonality}.

\begin{question} \label{vraag}
Does $\Q(X_1(N))$ always contain a modular unit of degree equal to the $\Q$-gonality?
\end{question}
It does not suffice to restrict to rational cusps ($C_i$'s of degree 1) because then $N=36$ would
be the first counter example.
Question~\ref{vraag} may seem likely at first sight, after all, it is true for $N \leqslant 40$.
However, we do not conjecture it because the function $f_{\rm BN} \in \C(X_1(29))$ from Remark~\ref{sharpCompare} is not
a modular unit over $\C$, but unlike Conjecture~\ref{klomp}, there is no compelling reason to restrict Question~\ref{vraag} to $\Q$.


\section{Lower bound for the gonality} \label{lower}
Task~3 is equivalent to showing that the Riemann-Roch space $H^0(C,D)$ is $\mathbb{F}_p$ for
every divisor $D \geqslant 0$ of degree $<{\rm deg}(g_N)$.  This is a finite task, because over $\mathbb{F}_p$, the number
of such $D$'s is finite.
For $N=37$, the $\Q$-gonality is 18, and the number of $D$'s over $\mathbb{F}_2$ with $D \geqslant 0$
and ${\rm deg}(D)<18$ is far too large to
be checked one by one on a computer. So we will need other methods to prove:

\begin{theorem} \label{N37}  Let $f \in \mathbb{F}_2(X_1(37)) - \mathbb{F}_2$.  Then ${\rm deg}(f) \geqslant 18$.
\end{theorem}

\begin{definition}
Let $f \in K(X_1(N))$.  The support $\supp(\Div(f))$ is $\{ P \in X_1(N)_K \, |$ $v_P(f) \neq 0\}$,
i.e., the set of places where $f$ has a non-zero valuation (a root or a pole).
Let ${\rm mdeg}_K(f)$ denote ${\rm max}\{ {\rm deg_K}(P) \, | \, P \in \supp(\Div(f)) \}$.  Likewise, if $D = \sum n_i P_i$ is a divisor,
then ${\rm mdeg}_K(D) := {\rm max}\{ {\rm deg}_K(P_i) \, | \, n_i \neq 0\}$.
\end{definition}

\noindent {\bf Overview of the proof of Theorem~\ref{N37}:}  \\[2pt]
We split the proof in two cases: Section~\ref{md1}
will prove Theorem~\ref{N37} for the case ${\rm mdeg}(f)=1$.
Section~\ref{md2} will introduce notation, and prove Theorem~\ref{N37} for
the case ${\rm mdeg}(f) > 1$.
(Task~3 for the remaining $N \leqslant 40$ is similar to Section~\ref{md2} but easier, and will be discussed in Section~\ref{Task3}.)

{\subsection{The $\F_2$ gonality of $X_1(37)$}
In \cite{scriptie} there are already tricks for computing the $\F_p$ gonality 
in a computationally more efficient way then the brute force method from earlier papers.
These tricks were not efficient enough to compute the $\F_2$ gonality of $X_1(37)$.
However, by subdividing the problem, treating one part with lattice reduction techniques, and the
other part with tricks from \cite{scriptie},
the case $N=37$ becomes manageable on a computer.
We divide the problem as follows:
\begin{prop}\label{prop:subdivision}
If there is a $g\in \F_2(X_1(37)) - \F_2$ with $\deg(g) \leqslant 17$ then there is an $f \in \F_2(X_1(37)) - \F_2$ with $\deg(f) \leqslant  17$ that satisfies at least one of the following conditions:
\begin{enumerate}
\item $\mdeg(f) = 1$ \label{enum:subdiv1}
\item all poles of $f$ are rational cusps, and $f$ has $\geqslant 10$ distinct poles. \label{enum:subdiv2}
\item $f$ has a pole at $\geqslant 5$ rational cusps and at least one non-rational pole. \label{enum:subdiv3}
\end{enumerate}
\end{prop}
\begin{proof}
$X_1(37)$ has 18 $\F_2$-rational places, all of which are cusps.
View $g$ as a morphism $X_1(37)_{\F_2} \to \P^1_{\F_2}$. For all $h \in \aut (\P^1_{\F_2})$ we have
$\deg(g) = \deg(h \circ g)$.  If there is $h \in \aut (\P^1_{\F_2})$
such that $\mdeg (h \circ g) =1$ then take $f=h\circ g$ and we are done.
Now assume that such $h$ does not exist. Then at least two
of the three sets $g^{-1}(\{0\}),g^{-1}(\{1\}),g^{-1}(\{\infty\})$ contain a non-rational place.
If all three do, then the one with the most rational cusps has at least $18/\#\P^1(\F_2)=6>5$
rational cusps and we can take $f=h\circ g$ for some $h \in \aut (\P^1_{\F_2})$.
Otherwise we can assume without loss of generality that $g^{-1}(\{\infty\})$ only contains rational cusps.
If $g^{-1}(\{\infty\})$ contains at least $10$ elements
then we can take $f=g$. If $g^{-1}(\{\infty\})$ contains at most $9$ elements
then $g^{-1}(\{0\}) \cup g^{-1}(\{1\})$ contains at least 
$18 - 9 =9$ rational cusps, so either $g^{-1}(\{0\})$ or $g^{-1}(\{1\})$ contains at least $5$, and we can take $f=1/g$ or $f=1/(1-g)$.
\end{proof}

\subsection{The case $N = 37$ and mdeg = 1} \label{md1}

\begin{prop}
\label{prop2}
Every $f\in \F_2(X_1(37)) - \F_2$ with $\mdeg(f)=1$
has $\deg(f) \geqslant 18$.
\end{prop}
\begin{proof}
Let $M=\Z^{X_1(37)(\F_2)} \subset \Div(X_1(37)_{\F_2})$ be the set of all divisors $D$ with
$\mdeg(D)=1$.  Let $N=\ker( M \to \Pic X_1(37)_{\F_2})$, i.e.
principal divisors in $M$. Magma can compute
$N$ directly from its definition, an impressive feat considering the size of the equation!
First download the file \verb+X1_37_AFF.m+
from our web-page \cite{URLgonality}. It contains the explicit equation for $X_1(37)$ over $\F_2$, and assigns it to
\verb+AFF+
with the Magma command AlgorithmicFunctionField.
\begin{verbatim}
> load "X1_37_AFF.m";
> plc1 := Places(AFF, 1); //18 places of degree 1, all cusps.
> M := FreeAbelianGroup(18);  gen := [M.i : i in [1..18]];
> ClGrp, m1, m2 := ClassGroup(AFF); //takes about 3 hours.
> N := Kernel(Homomorphism(M, ClGrp, gen, [m2(i) : i in plc1]));
\end{verbatim}
Let $\norm \noarg 1$ and $\norm \noarg 2$ be the standard $1$ and $2$ norm on $M$ with respect to the basis $X_1(37)(\F_2)$
(i.e. \verb+plc1+).
For a divisor $D\in N$ with $D = \Div(g)$ we have $\deg(g) = \frac 12 \norm D 1$.
So we need to show that $N$ contains no non-zero $D$ with $\norm D 1  \leqslant  2\cdot 17$.
The following calculation shows that $N$ contains no divisors $D \neq 0$
with $\norm D 2 ^2 \,\leqslant 2(14^2+3^2)=410$ and
$\frac 1 2 \norm D 1 \, \leqslant 17$.
\begin{verbatim}
> //Convert N to a more convenient data-structure.
> N := Lattice(Matrix( [Eltseq(M ! i) : i in Generators(N)] ));
> SV := ShortVectors(N,410);
> Min([&+[Abs(i) : i in Eltseq(j[1])]/2 : j in SV]);
18 1
\end{verbatim}
From this we can conclude two things.
First, there is a function $f$ of degree $18$ with $\mdeg(f)=1$.
We already knew that from our LLL search of ${\rm div}( \CuspGroup(37) )$, but this is nevertheless useful
for checking purposes (see Remark~\ref{remarkcheck} below).
Second, if there is a non-constant function $f$ of degree $\leqslant 17$ and
$\mdeg(f)=1$ then $\norm {\Div f} 2 ^2 \, > 2(14^2+3^2)$ so either $f$ or $1/f$ must
have a pole of order $\geqslant 15$ at a rational point. Then
either $f$ or $1/f$ is in a Riemann-Roch space $H^0(X_1(N)_{\F_2},15p+q+r)$ with $p,q,r$ in $X_1(37)(\F_2)$.
Since the diamond operators act transitively on $X_1(37)(\F_2)$
we can assume without loss of generality that $p$ is the first element of $X_1(37)(\F_2)$ returned by Magma.
The proof of the proposition is then completed with the following computation:

\begin{verbatim}
> p := plc1[1];
> Max([Dimension(RiemannRochSpace(15*p+q+r)) : q,r in plc1]);
1 1 \end{verbatim} \vspace{-10pt} \end{proof}
}

\begin{remark} \label{remarkcheck}
Computer programs could have bugs, so it is reasonable to ask if
Magma really did compute a proof of Proposition~\ref{prop2}.
The best way to check this is with independent verification, using other computer algebra systems.

We computed ${\rm div}(f_k)$, for $k=2,\ldots,\lfloor 37/2 \rfloor + 1$, in Maple with two separate methods. One is based
on determining root/pole orders by high-precision floating point evaluation at points close to the cusps. The second
method is based on Puiseux expansions.  The resulting divisors are the same.
Next, we searched the $\Z$-module spanned by these divisors for vectors with a low 1-norm.
Maple and Magma returned the same results, but what is important to note is that this search (in characteristic 0)
produced the same vectors as the divisors of degree-18 functions (in characteristic 2) that Magma found in the
computation for Proposition~\ref{prop2}.

We made similar checks throughout our work. 
Magma's RiemannRochSpace command never failed to find a function whose existence was known from a computation
with another computer algebra system. The structure of Magma's ClassGroup also matched results from computations in Sage and Maple.

The key programs that the proofs of our lower bounds depend on are 
Magma's RiemannRochSpace program (needed for all non-trivial $N$'s), and ClassGroup program (needed for $N=37$).
We have thoroughly tested these programs, and are confident that they compute correct proofs.
\end{remark}

\subsection{The case $N = 37$ and mdeg $>$ 1} \label{md2}
{
It remains to treat cases 2 and 3 of Proposition \ref{prop:subdivision}.
Let $S_2 \subseteq\F_2(X_1(37)) - \F_2$ be the set of all functions $f$ with $\deg(f) \leqslant 17$ such 
that all poles of $f$ are rational and $f$ has
at least $10$ distinct poles.
Similarly let $S_3 \subseteq\F_2(X_1(37)) -  \F_2$ be the set of all functions $f$ with $\deg(f) \leqslant 17$ such that $f$ has a pole at at least 
$5$ distinct rational points and a pole at at least $1$ non-rational point.
To complete the proof of Proposition \ref{prop:subdivision} we need to show:
\begin{prop}\label{prop:emptySi}
The sets $S_2$ and $S_3$ are empty.
\end{prop}
We will prove this with Magma computations, using ideas similar to those in \cite{scriptie}. To main new idea is in the following definition:
\begin{definition} Let $C$ be a curve over a field $\F$ and $S\subseteq\F(C) - \F$ a set of non-constant functions.
We say that a that a set of divisors $A \subset \Div C$ {\em dominates}
$S$ if for every  $f \in S$ there is a $D\in A$ such that $f \in \aut(\P^1_\F) H^0(C,D) \aut(C)$
(i.e. $f = g \circ f' \circ h$ for some $g \in  \aut(\P^1_\F)$, $f' \in H^0(C,D)$, and $h \in \aut(C)$).
\end{definition}
It follows directly from this definition that $$S \subseteq \bigcup_{D\in A} \aut(\P^1_\F) H^0(C,D) \aut(C)$$
and hence:
\begin{prop}\label{prop:empty}
Let $C$ be a curve over a field $\F$, $S\subseteq\F(C) - \F$ and $A \subset \Div C$.
Suppose that $A$ dominates $S$, and that:
\begin{equation}
\forall_{D \in A} \ S \cap  \aut(\P^1_\F) H^0(C,D) \aut(C) = \emptyset. \label{eq:hypothesis1}
\end{equation} Then $S=\emptyset$.
\end{prop}

\begin{proof}[Proof of Proposition~\ref{prop:emptySi}]
Appendix~\ref{sec:calculation37} gives
two sets $A_2$ and $A_3$ that dominate $S_2$ and $S_3$ respectively.
The Magma computations given there show that
$$ \forall_{D \in A_2 \cup A_3}
\min \{ \deg(f) \, | \,  f \in H^0(C,D) - \F_2 \} \geqslant 18$$
where $C = X_1(37)_{\F_2}$.
Since $\deg(f)$ is invariant under the actions of $\aut(\P^1_\F)$ and $\aut(C)$ it follows (for $i=2,3$ and $D\in A_i$) that $S_i \cap  
\aut(\P^1_\F) H^0(C,D) \aut(C)=\emptyset$ so we can apply Proposition~\ref{prop:empty}.
\end{proof}

\subsection{The cases $N \leqslant 40$ and $N \neq 37$} \label{Task3}
Subdividing the problem into three smaller
cases as in Proposition~\ref{prop:subdivision} was not necessary for the other $N \leqslant 40$.
Instead we used an easier approach which is similar to the case $N = 37$ and mdeg $>$ 1.


For an integer $N$ let $p_N$ denote the smallest prime $p$ such that $p \nmid N$.
Let $d_N = \deg(g_N)$ denote the degree of the lowest degree function we found for $N$ (Section~\ref{Task1}
or online~\cite{URLgonality}).
Now in order to prove $\gon_\Q (X_1(N)) \geqslant d_N$ we will prove $\gon_{\F_{p_N}} (X_1(N)) \geqslant d_N$.
We have done this by applying Proposition~\ref{prop:empty} directly with $S$ the set of all functions of degree
$< d_N$. To verify hypothesis~(\ref{eq:hypothesis1}) from Proposition \ref{prop:empty} with a computer for $A=\Div^+_{d_N 
-1}(X_1(N)_{\F_{p_N}})$ (i.e. all effective divisors of degree $d_N-1$) was unfeasible in a lot of cases. Instead we
used the following proposition to obtain a smaller set $A$ of divisors that still dominates all functions of degree $< d_N$.

\begin{prop}\label{prop:optimisation1}
Let $C$ be a curve over a finite field $\F_q$ and $d$ an integer. Let $n:=\lceil\# C(\F_q)/(q+1) \rceil$ and $$D=\sum_{p\in C(\F_q)} p$$ then
$$A:=\Div^+_{d-n}(C)+D=\set{ s'+D \mid s' \in \Div^+_{d-n}(C)}$$ dominates all functions of degree $\leqslant d$.
\end{prop}

\begin{proof}
For all $f\colon C \to \P^1_{\F_q}$ we have $f(C(\F_q))\subseteq \P^1(\F_q)$.  By the pigeon hole principle, there
is a point $p$ in $\P^1(\F_q)$ whose pre-image under $f$ has at least $n$ points in $C(\F_q)$. Moving $p$ to $\infty$ with
a suitable $g \in \aut(\P^1_{\F_q})$, the function $g \circ f$ has at least $n$ distinct poles in $C(\F_q)$.
So if $\deg(f) \leqslant d$
then $\Div(g \circ f) \geqslant -s -D$ for some $s \in \Div^+_{d-n}(C)$.
\end{proof}

Proposition~\ref{prop:optimisation1} reduces the number of divisors to check, but increases their degrees.
However, for our case $C=X_1(N)$ the gonality is generally much lower then the genus, 
so the Riemann-Roch spaces from equation~(\ref{eq:hypothesis1}) are still so small that it is
no problem to enumerate all their elements, and compute their degrees to show 
$S \cap  \aut(\P^1_\F) H^0(C,D) \aut(C) = \emptyset$.

As a further optimization we can make
$A$ even smaller by using the orbits
under diamond operators.
The Magma computations~\cite{URLgonality} show that hypothesis~(\ref{eq:hypothesis1})
in Proposition \ref{prop:empty} is satisfied for $S$, the set of functions of degree $< d_N $ in
$\F_{p_N}(X_1(N)) - \F_{p_N}$, and $A$, an explicit set 
of divisors dominating $S$.

Despite all our tricks to reduce the number of divisors, the number of divisors for $N=37$ (due to its high gonality)
remained far too high for our computers, specifically, divisors consisting of rational places.
We handled those by using the relations between rational places in the Jacobian.
That idea (worked out in Section~\ref{md1})
allowed us to complete $N=37$ and thus all $N \leqslant 40$.

\section{Patterns in the gonality data}
\label{Patterns}
\begin{definition}\label{def:improvement}
Let $\Gamma \subseteq \PSL_2(Z)$ be a congruence subgroup
and $X(\Gamma):= \halfplane^* /\Gamma$ be the corresponding modular curve over $\C$.
The {\em improvement factor} of a function $f \in \C(X(\Gamma)) - \C$ is the ratio $$[\PSL_2(\Z) :\Gamma]/\deg(f)=\deg(j)/\deg(f).$$
\end{definition}
The definition is motivated by a well known bound from Abramovich:
\begin{theorem}[\cite{abramovich}]
\label{BoundAbr}
$$\gon_\C(X(\Gamma)) \geqslant \frac{\lambda}{24} [\PSL_2(\Z) :\Gamma]$$
where $\lambda = 0.21$ (a lower bound for $\lambda_1$).
Kim and Sarnak \cite{KimSarnak} improved this to $\lambda={975}/{4096}$
(which is close to the $1/4$ from Selbergs conjecture).
\end{theorem}
The theorem says that an improvement factor can not exceed $24/\lambda \approx 100.825$, for any $\Gamma$, even over $\C$.
To compare this with $X_1(N)$ (over $\Q$), we plotted the improvement 
factors of our $g_N$'s from~\cite{URLgonality}.
This revealed a remarkable structure: \\

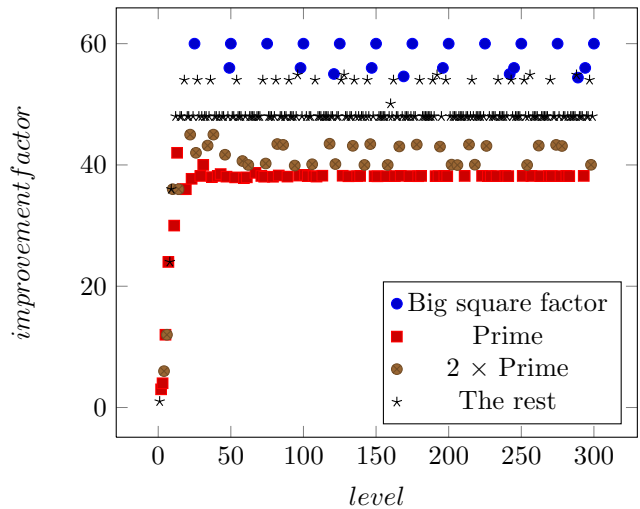
\begin{tikzpicture}
    \begin{axis}[
        xlabel=$level$,
        ylabel=$improvement factor$,
        legend pos=south east,
        only marks]
        
    \addplot 
       plot coordinates {
        (25, 60)
        (49, 56)
        (50, 60)
        (75, 60)
        (98, 56)
        (100, 60)
        (121, 55)
        (125, 60)
        (147, 56)
        (150, 60)
        (169, 273/5)
        (175, 60)
        (196, 56)
        (200, 60)
        (225, 60)
        (242, 55)
        (245, 56)
        (250, 60)
        (275, 60)
        (289, 272/5)
        (294, 56)
        (300, 60)
    };
    \addlegendentry{Big square factor}

    \addplot 
        plot coordinates {
            (2, 3)
            (3, 4)
            (5, 12)
            (7, 24)
            (11, 30)
            (13, 42)
            (17, 36)
            (19, 36)
            (23, 264/7)
            (29, 420/11)
            (31, 40)
            (37, 38)
            (41, 420/11)
            (43, 77/2)
            (47, 1104/29)
            (53, 1404/37)
            (59, 870/23)
            (61, 1860/49)
            (67, 1122/29)
            (71, 420/11)
            (73, 1332/35)
            (79, 1560/41)
            (83, 574/15)
            (89, 495/13)
            (97, 1568/41)
            (101, 5100/133)
            (103, 5304/139)
            (107, 954/25)
            (109, 495/13)
            (113, 6384/167)
            (127, 8064/211)
            (131, 572/15)
            (137, 1564/41)
            (139, 420/11)
            (149, 3700/97)
            (151, 11400/299)
            (157, 12324/323)
            (163, 1107/29)
            (167, 13944/365)
            (173, 3741/98)
            (179, 267/7)
            (181, 420/11)
            (191, 9120/239)
            (193, 2328/61)
            (197, 4851/127)
            (199, 6600/173)
            (211, 420/11)
            (223, 1184/31)
            (227, 8588/225)
            (229, 8740/229)
            (233, 3016/79)
            (239, 420/11)
            (241, 29040/761)
            (251, 420/11)
            (257, 33024/865)
            (263, 5764/151)
            (269, 3015/79)
            (271, 18360/481)
            (277, 12788/335)
            (281, 420/11)
            (283, 40044/1049)
            (293, 10731/281)
        };
    \addlegendentry{Prime}
        \addplot 
            plot coordinates {
                (4, 6)
                (6, 12)
                (10, 36)
                (14, 36)
                (22, 45)
                (26, 42)
                (34, 216/5)
                (38, 45)
                (46, 792/19)
                (58, 1260/31)
                (62, 40)
                (74, 684/17)
                (82, 1260/29)
                (86, 693/16)
                (94, 3312/83)
                (106, 1404/35)
                (118, 87/2)
                (122, 5580/139)
                (134, 561/13)
                (142, 40)
                (146, 999/23)
                (158, 40)
                (166, 861/20)
                (178, 5940/137)
                (194, 1764/41)
                (202, 7650/191)
                (206, 7956/199)
                (214, 477/11)
                (218, 3564/89)
                (226, 1596/37)
                (254, 24192/605)
                (262, 6435/149)
                (274, 14076/325)
                (278, 345/8)
                (298, 8325/208)
            };
    \addlegendentry{2 $\times$ Prime }
            \addplot 
                plot coordinates {

(1, 1)
(8, 24)
(9, 36)
(12, 48)
(15, 48)
(16, 48)
(18, 54)
(20, 48)
(21, 48)
(24, 48)
(27, 54)
(28, 48)
(30, 48)
(32, 48)
(33, 48)
(35, 48)
(36, 54)
(39, 48)
(40, 48)
(42, 48)
(44, 48)
(45, 48)
(48, 48)
(51, 48)
(52, 48)
(54, 54)
(55, 48)
(56, 48)
(57, 48)
(60, 48)
(63, 48)
(64, 48)
(65, 48)
(66, 48)
(68, 48)
(69, 48)
(70, 48)
(72, 54)
(76, 48)
(77, 48)
(78, 48)
(80, 48)
(81, 54)
(84, 48)
(85, 48)
(87, 48)
(88, 48)
(90, 54)
(91, 48)
(92, 48)
(93, 48)
(95, 48)
(96, 384/7)
(99, 48)
(102, 48)
(104, 48)
(105, 48)
(108, 54)
(110, 48)
(111, 48)
(112, 48)
(114, 48)
(115, 48)
(116, 48)
(117, 48)
(119, 48)
(120, 48)
(123, 48)
(124, 48)
(126, 54)
(128, 384/7)
(129, 48)
(130, 48)
(132, 48)
(133, 48)
(135, 54)
(136, 48)
(138, 48)
(140, 48)
(141, 48)
(143, 48)
(144, 54)
(145, 48)
(148, 48)
(152, 48)
(153, 48)
(154, 48)
(155, 48)
(156, 48)
(159, 48)
(160, 1152/23)
(161, 48)
(162, 54)
(164, 48)
(165, 48)
(168, 48)
(170, 48)
(171, 48)
(172, 48)
(174, 48)
(176, 48)
(177, 48)
(180, 54)
(182, 48)
(183, 48)
(184, 48)
(185, 48)
(186, 48)
(187, 48)
(188, 48)
(189, 54)
(190, 48)
(192, 384/7)
(195, 48)
(198, 54)
(201, 48)
(203, 48)
(204, 48)
(205, 48)
(207, 48)
(208, 48)
(209, 48)
(210, 48)
(212, 48)
(213, 48)
(215, 48)
(216, 54)
(217, 48)
(219, 48)
(220, 48)
(221, 48)
(222, 48)
(224, 48)
(228, 48)
(230, 48)
(231, 48)
(232, 48)
(234, 54)
(235, 48)
(236, 48)
(237, 48)
(238, 48)
(240, 48)
(243, 54)
(244, 48)
(246, 48)
(247, 48)
(248, 48)
(249, 48)
(252, 54)
(253, 48)
(255, 48)
(256, 384/7)
(258, 48)
(259, 48)
(260, 48)
(261, 48)
(264, 48)
(265, 48)
(266, 48)
(267, 48)
(268, 48)
(270, 54)
(272, 48)
(273, 48)
(276, 48)
(279, 48)
(280, 48)
(282, 48)
(284, 48)
(285, 48)
(286, 48)
(287, 48)
(288, 384/7)
(290, 48)
(291, 48)
(292, 48)
(295, 48)
(296, 48)
(297, 54)
(299, 48)
                };
  \addlegendentry{The rest}
 \end{axis}
\end{tikzpicture}

What immediately pops out is that our best improvement factor is often 48 (in 151 out of 300 levels $N$).
Levels $N>9$ with an improvement factor $<48$ are either of the form $N = p$ or $N = 2p$ for a prime $p$.
For prime levels, our improvement factor converges to $420/11$.

Levels of the form $N = kp^2$ with $p>3$ prime stand out in the graph, with improvement factors significantly higher than 48.
To explain this, first observe that improvement factors for $kp^2$ are $\geqslant$ those of $p^2$ because:

\begin{remark}
If $\Gamma \subseteq \Gamma'$ are two congruence subgroups, $\pi: X(\Gamma) \to X(\Gamma')$ denotes the quotient map and $f \in \C(X(\Gamma'))$ then 
$f$ and $f\circ \pi$ have the same improvement factor. So improvement factors for $X(\Gamma')$ can not exceed those for $X(\Gamma)$.
\end{remark}

It remains to explain the high observed improvement factors at levels $N = p^2$:

\begin{center}
\begin{tabular}{|c|c|c|c|c|c|}
\hline
level & $5^2$ & $7^2$ & $11^2$ & $13^2$ & $17^2$ \\
improvement & 60 & 56 & 55 & 54 $\frac 3 5$ & $54 \frac 2 5$ \\ \hline   
\end{tabular}
\end{center}

The best (lowest degree, highest improvement factor) modular units $g_N$ we found for these five cases turned out to be
invariant under a larger congruence subgroup $\Gamma_0(p^2) \cap \Gamma_1(p) \supseteq \Gamma_1(p^2)$.
Now $$\Gamma_0(p^2) \cap \Gamma_1(p) = \left[ \begin{array}{rr} 1 & 0 \\ 0 & p\end{array} \right] \Gamma(p) \left[ \begin{array}{rr} 1 & 0 \\ 0 & p\end{array} \right]^{-1}.$$
This suggests to look at $X(p)$ to find high improvement factors for $X_1(p^2)$.

\section{Points of degree 5, 6, 7, and 8 on $X_1(N)$}
The values of $N$ for which the curve $X_1(N)$ has infinitely many places of degree $d$ over $\Q$
are known for $d=1$ (\cite{mazur}: $N=$1..10,12), $d=2$ (\cite{d2}: $N=$1..16,18), $d=3$ (\cite{d3}: $N=$1..16,18,20) and $d=4$
(\cite{d4}: $N=$1..18,20..22,24).
In this section, we extend this to $d \leqslant 8$.

\begin{theorem}\label{rat_points} Let $5 \leqslant d \leqslant 8$, then
$X_1(N)$ has infinitely many places of degree $d$  over $\Q$ if and only if
\begin{itemize}
\item for $d=5$: $N \in \{1,\ldots,25\} - \{23\}$.
\item for $d=6$: $N \in \{1,\ldots,30\} - \{23, 25, 29\}$.
\item for $d=7$: $N \in \{1,\ldots,30\} - \{25, 29\}$.
\item for $d=8$: $N \in \{1,\ldots,28, 30, 32, 36\}$.
\end{itemize}
\end{theorem}  
The case $X_1(25)$ is the most interesting
because its set of non-cuspidal places of degree $d = 6, 7$ is finite\footnote{and not empty, we found an example for $d=6$ and $d=7$ \cite{URLgonality}}
even though this exceeds the $\Q$-gonality of $X_1(25)$!
The remainder of this section contains the proof of Theorem~\ref{rat_points}.

\begin{lemma}\label{finitely_many_deg_d_criterion}
\begin{enumerate}
\item Let $C/\Q$ be a curve. If $C$ has a function $f$ over $\Q$
of degree~$d$ then $C$ has infinitely many places of degree $d$ over $\Q$.
\item If the Jacobian $J(C)(\Q)$ is finite, then the converse holds as well. To be precise,
if $C$ has more than $\# J(C)(\Q)$ places of degree $d$, then $\Q(C)$ contains a function of degree $d$.
\item If $N \leqslant 66$ and $N \neq 37,43,53,57,58,61,63,65$ then $J_1(N)(\Q)$ is finite.
\item If $N > 66$ or $N = 37,43,53,57,58,61,63,65$ then $X_1(N)$ has finitely many places of degree $\leqslant 8$.
\end{enumerate}
\end{lemma}
\begin{proof}

\begin{enumerate}
\item Hilbert's irreducibility theorem shows that there are infinitely many places of degree $d$
among the roots of $f-q=0$, $q \in \Q$.
\item If $n = \#J(C)(\Q) < \infty$ and $P_1,\ldots,P_{n+1}$ are distinct places of degree $d$, then by the pigeon hole principle,
there exist $i \neq j$ with $P_i-P_1 \sim P_j - P_1$.
The function giving this linear equivalence has degree $d$.
\item Magma has a provably correct algorithm to determine if $L(J_1(N),1)$ is $0$ or not. It shows $L(J_1(N),1)\neq 0$ for each
$N$ in item~3. By a result of Kato this implies that $J_1(N)(\Q)$ has rank zero and hence is finite.
\item   
The case $N=58$ follows from the map $X_1(58) \to X_1(29)$ and the fact that $X_1(29)$ has only finitely many points of degree $< 11$ (by items~3,~2
and Table~1).
$\gon_{\Q}(X_1(37))=18$, and a computation (available at \cite{URLgonality}) that followed the strategy from Appendix~A
showed $\gon_{\Q}(X_1(N)) \geqslant 17$ for $N = 43, 53, 57$
(these bounds are surely not sharp, but proving a higher bound is a lot more work).
Theorem~\ref{BoundAbr} implies $\gon_{\Q}(X_1(N)) \geqslant 17$ for $N = 61,63,65$ and every $N>66$.
Now item~4 follows from the main theorem of \cite{frey} which states that a
curve $C/\Q$ with $C(\Q) \neq \emptyset$ has finitely many places of degree $< \gon_{\Q}(C)/2$.
\end{enumerate}
\end{proof}

\noindent Items 4, 3, 2, and 1 of Lemma~\ref{finitely_many_deg_d_criterion} reduce Theorem~\ref{rat_points} step by step to:
\begin{prop}\label{function_degrees}
Let $5 \leqslant d \leqslant 8$. Then $X_1(N)$ has a function over $\Q$ of degree $d$
if and only if $N$ is as in Theorem~\ref{rat_points}.
\end{prop}
\begin{proof} Let $5 \leqslant d \leqslant 8$.
For each $N$ listed in Theorem~\ref{rat_points}, our divisor data \cite{URLgonality} makes it easy to find an
explicit 
modular unit
(Section~\ref{CuspFunctions}) of degree $d$.
So it suffices to show that there are no functions of degree $d$ for the other $N$'s.
 \begin{itemize}
  \item $N > 40$ and $N \neq 42, 44, 46, 48$: Theorem~\ref{BoundAbr} implies $\gon_{\Q}(X_1(N)) > 8$.
  \item $N \leqslant 40$ or $N = 42, 44, 46, 48$ and $(N,d) \neq (25,6),(25,7)$: For $N \leqslant 40$ see Table~1. Similar computations
   (based on Proposition~\ref{prop:optimisation1}, available at \cite{URLgonality}) show that $\gon_\Q(X_1(N)) > 8$ for $N = 42, 44, 46, 48$.
 \item $N=25$ and $d=6,7$: We prove this by verifying conditions 1--5 of Proposition~\ref{technical_proposition} below
 with $C = X_1(25)$, $d=6,7$ and $p=2$.
 \begin{itemize}
 \item[1. ] The rank of $J_1(25)(\Q)$ is 0 and $\# J_1(25)(\mathbb F_3)=2503105$ is odd. So $\# J_1(25)(\Q)$ is finite and odd and hence $J_1(25)(\Q)
 \hookrightarrow J_1(\F_2)$.
 \item[2,3] We verified this using a Magma computation (files at \cite{URLgonality}).
 \item[4. ] 
We need to show that $W_5^1(\Q) \to W_5^1(\F_2)$ and $W_6^1(\Q) \to W_6^1(\F_2)$ are surjective. $W_5^1(\Q) \neq \emptyset$ by Table~1.
A Magma computation showed that $\#W_5^1(\F_2)  = 1$ and that
 every element in $W_6^1(\F_2)$ is of the form $D+P$ where $P \in X_1(25)(\F_2)$ and
  $D$ is the unique element of $W_5^1(\F_2)$. Such $D+P$ lift to $W_6^1(\Q)$ because $X_1(25)(\Q) \to X_1(25)(\F_2)$ is surjective.
 \item[5. ] This is true because $X_1(25)$ has no places of degree 2 over $\F_2$, and the rational places over $\F_2$ are exactly the 10 cusps
  that come from the rational cusps in $X_1(25)(\Q)$.
  \end{itemize}
  \end{itemize}
 \end{proof}


For $N \leqslant 40$, applying a {\tt ShortVectors}-search to our divisor data \cite{URLgonality}
shows that $\Q(X_1(N))$ has a function of degree $d \geqslant \gon_{\Q}(X_1(N))$
if $(N,d) \not\in S = \{(25,6)$, $(25,7)$, $(32,9)$, $(33,11)$, $(35,13)$, $(39,15)$, $(40,13)\}$.
The search also showed that there are no modular units with $(N,d) \in S$.
Ruling out degree-$d$ functions other than modular units is more work:
 
\begin{prop}\label{technical_proposition}
Let $C/\Q$ with $C(\Q) \neq \emptyset$ be a smooth projective curve with good reduction at a prime $p$. Let $W_d^r(K)$
denote the closed subscheme of $\Pic^{d} C(K)$
corresponding to the line bundles $\mathcal L$ of degree $d$ whose global sections form a $K$-vectorspace
of dimension $\geqslant r+1$. Suppose that:
\begin{enumerate}
\item $J(C)(\Q) \to J(C)(\F_p)$ is injective.
\item $\F_p(C)$ contains no functions of degree $d$.
\item $W_d^2(\F_p) = \emptyset$.
\item $W^1_{d-i}(\Q)  \to W^1_{d-i}(\F_p)$ is surjective for all $1 \leqslant i \leqslant d - \gon_{\F_p}(C)$.
\item $C^{(i)}(\Q)  \to C^{(i)}(\F_p)$ is surjective for all $1 \leqslant i \leqslant  d-\gon_{\F_p}(C)$.
\end{enumerate}
Then $\Q(C)$ contains no functions of degree $d$.
\end{prop}
\begin{proof}
Item~1 and $C(\Q) \neq \emptyset$ imply that $\Pic^k C(\Q)$ to $\Pic^k C(\F_p)$ is injective for all $k$.
To show that $\Q(C)$ has no function of degree $d$ it suffices to show for all $\mathcal L$ in $W_d^1(\Q)$
that every $2$-dimensional subspace $V \subset \mathcal L(C)$ has a base point. 

Let $\mathcal L \in W_d^1(\Q)$. Item~3 implies $\dim_{\F_p} \mathcal L_{\F_p}(C_{\F_p})  = 2$ and so
$\dim_\Q \mathcal L(C)=2$. Let $D_{\F_p}$ be the divisor of basepoints of $\mathcal L_{F_p}$ and
let $i$ be its degree. Item~2 implies $i  \geqslant 1$ and because $\mathcal L_{\F_p}(-D_{\F_p}) \in W_{d-i}^1(\F_p)$ we have 
$i \leqslant d - \gon_{\F_p}(C)$. By item~5 there is a $D \in C^{(i)}(Q)$ that reduces to $D_{\F_p}$.
By of the injectivity of $\Pic^{d-i} C(\Q) \to \Pic^{d-i} C(\F_p)$, we know that $\mathcal L(-D)$
is the unique point lying above $\mathcal L_{F_p}(-D_{\F_p})$. Then item~4 gives
the following inequalities $$2 \leqslant \dim_\Q \mathcal L(-D)(C)  \leqslant \dim_\Q \mathcal L(C) = 2.$$
In particular, the unique $2$-dimensional $V \subset \mathcal L(C)$ has the points in $D$ as base points.
\end{proof}

\begin{remark} {\em \bf Extending Theorem~\ref{rat_points}}. 
With our divisor data \cite{URLgonality} we can quickly
find\footnote{This can be done for substantially larger $(N,d)$ as well.}
all $N$ for which $X_1(N)$ has a modular unit (and hence, infinitely many places) of degree $d=9$ over $\Q$.
The difficult part is to rule out infinitely many places of degree $9$ for the remaining $N$'s.
Take for example $N=57$.
Our proof that $\gon_\Q(X_1(57))) \geqslant 17$ was
already a lot of work (see~\cite{URLgonality}) but pushing this bound to 19 will be much more work still.

Another issue arises for $N=37$. The
Jacobian has positive rank, and the $\Q$-gonality is $18$ so we can not use Frey's theorem
to rule out infinitely many places of degree $9$. 
$J_1(37)$ has only one simple abelian sub-variety of positive rank, namely an elliptic curve $E$ isogenous to $X_0^+(37)$.
The question whether $X_1(37)$ has infinitely many places of degree $9$ is equivalent to the question whether $W_9^0(X_1(37))$ contains a translate of $E$.

Higher values of $d$ lead to additional problems.
If $d \leqslant 8$ then $X_1(N)_{\Q}$ has infinitely many places of degree $d$ if and only if it has a modular unit of degree~$d$,
but we do not expect that to continue for very large $d$'s.

\end{remark}

\appendix
\section[Appendix A]{Magma Calculations}
We use one custom function. It takes as input a divisor and gives as output the degrees of all non-constant functions in the 
associated Riemann-Roch space. 
{\small
\begin{verbatim}
function FunctionDegrees(divisor)
    constantField := ConstantField(FunctionField(divisor));
    space,map := RiemannRochSpace(divisor);
    return [Degree(map(i)) : i in space | map(i) notin constantField];
end function;
\end{verbatim}
}
We divide the computation according to {\em type}:
\begin{definition}
Write $D$ as $$\sum_{i=1}^kn_ip_i$$ with $p_i$ distinct places and $n_i \in \Z - \set{0}$ such that $(\deg(p_1),n_1)\geqslant  (\deg(p_2),n_2) \geqslant \cdots \geqslant (\deg(p_k),n_k)$ where
$\geqslant$ is the lexicographic ordering on tuples. Then
$type(D)$ is defined to be the ordered sequence of tuples
$$((\deg(p_1),n_1), (\deg(p_2),n_2),\ldots ,(\deg(p_k),n_k)).$$
If $\deg(p_i) =1$ for all $i$ then $(n_1,\ldots,n_k)$ is a shorter notation for ${\rm type}(D)$. 
\end{definition}
For example if $D=P_1+3P_2$ where $P_1$ is a place of degree $5$ and $P_2$ a place of degree 1 then $$type(D)=((5,1),(1,3)).$$
The type of a divisor is stable under the action of $\aut(C)$.
\subsection{The case $N = 37$ and mdeg $>$ 1} \label{sec:calculation37}
\subsubsection{Dominating the set $S_2$}
\label{SubS2}
Let \[
{\tt cuspsum} : = \sum_{p \in X_1(37)(\F_2)} p \]
(short for rational-cusp-sum) be the sum of all $\F_2$ rational places.
Then the set $$A_2' := \set{ {\tt cuspsum} + D \mid D=p_1+\cdots+p_7 \text{ with }  p_1,\ldots,p_7 \in X_1(37)(\F_2)}$$ dominates $S_2$.
However, $A_2'$ contains many divisors.
Using divisors of higher degree, of the form $k \cdot {\tt cuspsum} + \cdots$ for $k=1,2,3$ depending on type$(D)$,
we can dominate $S_2$ with much fewer divisors. To prove:
\begin{equation}
\min\{ \deg(f) \, | \, {f\in H^0(X_1(37)_{\F_2},{\tt cuspsum} + D) - \F_2} \} \geqslant 18 \label{eq:geq18}
\end{equation} for all ${\tt cuspsum}+D$ in $A'_2$ 
we divide the computation:
The table below list for each type$(D)$ (a partition of 7) from which
Magma calculation we can conclude inequality~($\ref{eq:geq18}$) for that type.

\begin{center}
\begin{tabular}{l|r}
$type(D)$ & calculation\\
\hline
(7), (6, 1) and (5, 2) & 1\\
(5, 1, 1), (4, 3), (4, 2, 1), (4, 1, 1, 1) and (3, 3, 1) & 2\\
(3, 2, 2) & 3 \\
(3, 2, 1, 1) and (3, 1, 1, 1, 1) & 2 \\
(2, 2, 2, 1), (2, 2, 1, 1, 1), (2, 1, 1, 1, 1, 1), (1, 1, 1, 1, 1, 1, 1) & 4
\end{tabular}
\end{center}

As in Section~\ref{md1}, start the computation by loading the file \verb+X1_37_AFF.m+.
Next, load the program \verb+FunctionDegrees+ and then run the following:

\begin{verbatim}
> //calculation 1
> p := plc1[1]; //diamond operators act transitively on X1(37)(F2)
> [Dimension(cuspsum + 6*p + 2*P) : P in plc1];
[ 1, 1, 1, 1, 1, 1, 1, 1, 1, 1, 1, 1, 1, 1, 1, 1, 1, 1 ]
\end{verbatim}
\begin{verbatim}
> //calculation 2
> Min(&cat[FunctionDegrees(2*cuspsum + 4*p + 2*P) : P in plc1]);
18 105
\end{verbatim}
\begin{verbatim}
> //calculation 3
> s := Subsets(SequenceToSet(plc1[2..18]),2);
> &cat[FunctionDegrees(cuspsum + 3*p + 2*(&+PQ)) : PQ in s];
[]
\end{verbatim}
\begin{verbatim}
> //calculation 4
> Min(FunctionDegrees(3*cuspsum - 4*p);
18 48
\end{verbatim}
The set $A_2$ in the proof of Proposition~\ref{prop:emptySi} is the set of divisors occurring in the four calculations above.
Calculation 4 used that if $f \in \F_2(X_1(37))$ has
$\deg(f) \leqslant 17$ then at least one of $f, f+1$ has an $\F_2$-rational
root since $\#X_1(37)(\F_2)=18$.

\subsubsection{Dominating the set $S_3$}
The set 
$$A'_3 := \set {{\tt cuspsum} +D \mid  D \geqslant 0, \ {\rm deg}(D)=12  \ {\rm with} \ \geqslant 1 {\rm \ nonrational \ place} }$$
dominates all functions in $S_3$.
This time we break up the computation into the following types where we use the following shorthand notation $$a(c,d):=\underbrace{(c,d),\ldots, (c,d)}_a$$
\begin{center}
\begin{tabular}{l|r}
type$(D)$ covered by calculation \#$c$ &$c$ \\
\hline
((12,1)) and ((11, 1),(1,1))			 & 1\\
((10,1),(1,2)) and ((10,1),(1,1),(1,1)) & 2\\
((9,1)(1,3)) & 3 \\
((9,1),(1,2),(1,1)) and ((9,1),(1,1),(1,1),(1,1)) & 4\\
((7,1),(1,5)), ((7,1),(1,4),(1,1)) and ((7,1),(1,3),(1,2)) & 5 \\
((7,1),(1,3),(1,1),(1,1)) and ((7,1),(1,2),(1,2),(1,1)) & 6 \\
((7,1),(1,2),3(1,1)) and ((7,1),5(1,1)) & 7\\
((6,2)) and ((6,1),(6,1)) & 8\\
((6,1),(1,6)), ((6,1),(1,5),(1,1)), ((6,1),(1,4),(1,2)), ((6,1),(1,3),(1,3)) & 9\\
((6,1),(1,4),2(1,1)), ((6,1),(1,3),(1,2),(1,1)), ((6,1),3(1,2)) &10\\
((6,1),2(1,2),2(1,1)),((6,1),(1,3),3(1,1)),((6,1),(1,2),4(1,1)),((6,1),6(1,1))&11
\end{tabular}
\end{center}
$X_1(37)_{\F_2}$ has no places of degrees 2--5 and 8.  So any non-rational place
contributes at least 6 to ${\rm deg}(D)$, a fortunate fact that reduces the number of divisors to a manageable level.
The Magma commands to cover these 11 cases are similar to those in Section~\ref{SubS2} and can be copied from~\cite{URLgonality}.
}

\begin{theorem}
The values in Table~1 are upper bounds for the gonality of $X_1(N)$ over $\Q$.  For $N \leqslant 40$ they are exact values.
\end{theorem}
\begin{proof}
The functions
listed at \cite{URLgonality} are explicit proofs for the upper bounds in Table~1. Section~\ref{lower}, Appendix A, and the
accompanying Magma files on \cite{URLgonality} prove that the bounds are sharp for $N \leqslant 40$.
\end{proof}


\bibliographystyle{alpha}

\end{document}